\documentclass[a4paper,12pt]{article}
\usepackage{amsmath, amsfonts, amssymb, amsthm}
\usepackage[english]{babel}

\newcounter{num}[section]

\newenvironment{theorem}
{\refstepcounter{num}%
\bigskip\noindent\nopagebreak[4]{\bf Theorem~\arabic{section}.\arabic{num}. }\it}
\newenvironment{proposition}
{\refstepcounter{num}%
\bigskip\noindent\nopagebreak[4]{\bf Proposition~\arabic{section}.\arabic{num}. }\it}
\newenvironment{lemma}
{\refstepcounter{num}%
\bigskip\noindent\nopagebreak[4]{\bf Lemma~\arabic{section}.\arabic{num}. }\it}
\newenvironment{remark}
{\refstepcounter{num}%
\bigskip\noindent\nopagebreak[4]{\bf Remark~\arabic{section}.\arabic{num}. }}

\newenvironment{example}
{\refstepcounter{num}%
\bigskip\noindent\nopagebreak[4]{\bf Example~\arabic{section}.\arabic{num}. }}

\newcommand{\Ss}{{\mathbf{S}}}
\newcommand{\V}{{\mathrm{V}}}
\newcommand{\Var}{{\mathrm{Var}}}
\newcommand{\Irr}{{\mathrm{Irr}}}

\newcommand{\lb}{{\langle}}
\newcommand{\rb}{{\rangle}}

\newcommand{\pr}{{\prime}}

\newcommand{\s}{{\sigma}}

\newcommand{\al}{{\alpha}}

\begin{document}
\title{On irreducible algebraic sets over linearly ordered semilattices II}
\author{Artem N. Shevlyakov}

\maketitle

\abstract{Equations over linearly ordered semilattices are studied. For any equation $t(X)=s(X)$ we find irreducible components of its solution set and compute the average number of irreducible components of all equations in $n$ variables.}
\section{Introduction}

This paper is devoted to the following problem. One can define a notion of an equation over a linearly ordered semilattice $L_l=\{a_1,a_2,\ldots,a_l\}$ (the formal definition of an equation is given below in the paper). A set $Y$ is {\it algebraic} if it is the solution set of some system of equations over $L_l$.  Let us consider an equation $t(X)=s(X)$ in $n$ variables over $L_l$, and $Y$ be the solution set of $t(X)=s(X)$. One can find algebraic sets $Y_1,Y_2,\ldots,Y_m$ such that $Y=\bigcup_{i=1}^m Y_i$. One can decompose each $Y_i$ into a union of other algebraic sets, etc. This process terminates after a finite number of steps and gives a decomposition of $Y$ into a union of {\it irreducible} algebraic sets $Y_i$ (the sets $Y_i$ are called the {\it irreducible components }of $Y$). Roughly speaking, irreducible algebraic sets are ``atoms'' which form any algebraic set. The size and the number of such ``atoms'' are important characteristics of the semilattices $L_l$, since there are connections between irreducible algebraic sets and universal theory of linearly ordered semilattices (see \cite{uni_Th_II}). Moreover, the number of irreducible components was involved in the estimation of lower bounds of algorithm complexity (see~\cite{ben-or} for more details).

In this paper we assume $n\leq l$ (i.e. the order of the semilattice $L_l$ is not less than the number of variables in $t(X)=s(X)$) and study (Section~\ref{sec:decomposition_properties}) the properties of algebraic sets over $L_l$. Precisely, for any equation $t(X)=s(X)$ in $n$ variables we count the number of irreducible components (see~(\ref{eq:Irr(k_1,k_2,n,l)})), and in Section~\ref{sec:average} we count the average number $\overline{\Irr}(n)$ of irreducible components of the solution sets of equations in $n$ variables.  

Remark that the current paper is the sequel of~\cite{shevl_LOS_irred}, where we solved the similar problems assuming $n>l$ (we discuss this case in Remark~\ref{rem:rem} below).

\section{Main definitions}

Let $L_l=\{a_1,a_2,\ldots,a_l\}$ be the linearly ordered semilattice of $l$ elements and $a_1<a_2<\ldots <a_l$. The multiplication in $L_l$ is defined by  $a_i\cdot a_j=a_{\min(i,j)}$. Obviously, the linear order on $L_l$ can be expressed by the multiplication as follows
\[
a_i\leq a_j\Leftrightarrow a_ia_j=a_i.
\]
A {\it term} $t(X)$ in variables $X=\{x_1,x_2,\ldots,x_n\}$ is a commutative word in letters $x_i$.

Let $\Var(t)$ be the set of all variables occurring in a term $t(X)$.
Following~\cite{uni_Th_II}, an {\it equation} is an equality of terms $t(X)=s(X)$. Below we consider inequalities $t(X)\leq s(X)$ as equations, since $t(X)\leq s(X)$ is the short form of $t(X)s(X)=t(X)$. Notice that we consider equations as {\it ordered pairs} of terms, i.e. the expressions $t(X)=s(X)$, $s(X)=t(X)$ are {\it different} equations. Let $Eq(n)$ denote the set of all equations in $X=\{x_1,x_2,\ldots,x_n\}$ variables (we assume that each $t(X)=s(X)\in Eq(n)$ contains the occurrences of all variables $x_1,x_2,\ldots,x_n$). An equation $t(X)=s(X)\in Eq(n)$ is said to be a {\it $(k_1,k_2)$-equation} if $|\Var(t)\setminus\Var(s)|=k_1$ and $|\Var(s)\setminus\Var(t)|=k_2$. For example, $x_1x_2=x_1x_3x_4$ is a $(1,2)$-equation. Let $Eq(k_1,k_2,n)\subseteq Eq(n)$ be the set of all $(k_1,k_2)$-equations in $n$ variables. Obviously,
\begin{equation}
Eq(n)=\bigcup_{(k_1,k_2)\in K_n}Eq(k_1,k_2,n),
\label{eq:Eq(n)}
\end{equation}
where 
\[
K_n=\{(k_1,k_2)\mid k_1+k_2\leq n\}\setminus\{(0,n),(n,0)\}.
\]

Each equation $t(X)=s(X)\in Eq(k_1,k_2,n)$ is uniquely defined by $k_1$ variables in the left part and by $k_2$ other variables in the right part (the residuary $n-k_1-k_2$ variables should occur in both parts of the equation). Thus, 
\[
\#Eq(k_1,k_2,n)=\binom{n}{k_1}\binom{n-k_1}{k_2}.
\]
By~(\ref{eq:Eq(n)}), one can compute that
\[
\#Eq(n)=3^n-2.
\]  

\begin{remark}
\label{rem:rem}
In this paper we consider only equations $t(X)=s(X)$ with $n\leq l$, i.e. the number of variables occurring in $t(X)=s(X)$ is not more than the order of the semilattice $L_l$. The case $n>l$ needs a completely different technic and was considered in~\cite{shevl_LOS_irred}. All main results of the current paper do not hold for the case $n>l$. 
\end{remark}

\bigskip

A point $P\in L_l^n$ is a {\it solution} of an equation $t(X)=s(X)$ if $t(P),s(P)$ define the same element in the semilattice $L_l$. By the properties of linearly ordered semilattices, a point $P=(p_1,p_2,\ldots,p_n)$ is a solution of $t(X)=s(X)$ iff there exist variables $x_i\in\Var(t)$, $x_j\in\Var(s)$ such that $p_i=p_j$ and $p_i\leq p_k$ for all $1\leq k\leq n$.  The set of all solutions of an equation $t(X)=s(X)$ is denoted by $\V(t(X)=s(X))$.

An arbitrary set of equations is called a {\it system}. The set of all solutions $\V(\Ss)$ of a system $\Ss=\{t_i(X)=s_i(X)\mid i\in I\}$ is defined as $\bigcap_{i\in I}\V(t_i(X)=s_i(X))$. A set $Y\subseteq L_l^n$ is called {\it algebraic over }$L_l$ if there exists a system $\Ss$ in $n$ variables with $\V(\Ss)=Y$. An algebraic set $Y$ is {\it irreducible} if $Y$ is not a proper finite union of other algebraic sets.  

\begin{proposition}\textup{(\cite{shevl_LOS_irred}, Proposition~2.2)}
Any algebraic set $Y$ over $L_l$ is a finite union of irreducible sets
\begin{equation}
Y=Y_1\cup Y_2\cup \ldots\cup Y_m,\quad Y_i\nsubseteq Y_j \mbox{ for all $i\neq j$},
\label{eq:union_Y_general}
\end{equation}
and this decomposition is unique up to a permutation of components.
\end{proposition} 

\bigskip

The subsets $Y_i$ from the union~(\ref{eq:union_Y_general}) are called the {\it irreducible components} of $Y$.

Let $Y$ be an algebraic set over $L_l$ defined by a system $\Ss(X)$. One can define an equivalence relation $\sim_Y$ over the set of all terms in variables $X$ as follows
\[
t(X)\sim_Y s(X)\Leftrightarrow t(P)=s(P) \mbox{ for any point $P\in Y$}.
\] 
The set of all $\sim_Y$-equivalence classes is called {\it the coordinate semilattice of $Y$} and denoted by $\Gamma(Y)$ (see~\cite{uni_Th_II} for more details). The following statement describes the coordinate semilattices of irreducible algebraic sets.

\begin{proposition}\textup{(\cite{shevl_LOS_irred}, Proposition~2.3)}
A set $Y$ is irreducible over $L_l$ iff $\Gamma(Y)$ is embedded into $L_l$ 
\label{pr:gamma_is_embedded_for_irr}
\end{proposition}  

\bigskip

There are different algebraic sets over $L_l$ with isomorphic coordinate semilattices. Such sets are called {\it isomorphic}. For example, the following sets
\[
Y_1=\V(\{x_1\leq x_2\leq x_3\}),\; Y_2=\V(\{x_3\leq x_2\leq x_1\})
\]
has the isomorphic coordinate semilattices
\[
\Gamma(Y_1)=\lb x_1,x_2,x_3\mid x_1\leq x_2\leq x_3\rb\cong L_3,
\]
\[
\Gamma(Y_2)=\lb x_1,x_2,x_3\mid x_3\leq x_2\leq x_1\rb\cong L_3.
\]
Thus, $Y_1,Y_2$ are isomorphic.

\newpage
\section{Example}

Let $n=3$, $l=3$. We have exactly $Eq(3)=3^3-2=25$ equations in three variables over $L_3$. The following table contains the  information about such equations over $L_3$. The second column contains systems which define irreducible components of the solution set of an equation in the first column. A cell of the table contains $\uparrow$ if an information in this cell is similar to the cell above.

\bigskip

Table 1.

\bigskip

\begin{tabular}{|c|c|c|}
\hline
Equations&Irreducible components (IC)&Number of IC\\
\hline
$x_1x_2x_3=x_1x_2x_3$&$x_1\leq x_2\leq x_3 \cup x_1\leq x_3\leq x_2\cup$ &$6$\\
&$x_2\leq x_1\leq x_3 \cup x_2\leq x_3\leq x_2\cup$&\\
&$x_3\leq x_1\leq x_2 \cup x_3\leq x_2\leq x_1$&\\
\hline
$x_1=x_1x_2x_3$,&$x_1\leq x_2\leq x_3\cup x_1\leq x_3\leq x_1$&$2$\\
$x_1x_2x_3=x_1$&&\\
\hline
$x_2=x_1x_2x_3$,&$\uparrow$&$2$\\
$x_1x_2x_3=x_2$&&\\
\hline
$x_3=x_1x_2x_3$,&$\uparrow$&$2$\\
$x_1x_2x_3=x_3$&&\\
\hline
$x_1=x_2x_3$,&$x_1=x_2\leq x_3\cup x_1=x_3\leq x_2$&$2$\\
$x_2x_3=x_1$&&\\
\hline
$x_2=x_1x_3$,&$\uparrow$&$2$\\
$x_1x_3=x_2$&&\\
\hline
$x_3=x_1x_2$,&$\uparrow$&$2$\\
$x_1x_2=x_3$&&\\
\hline
$x_1x_2=x_1x_3$,&$x_1\leq x_2\leq x_3\cup x_1\leq x_3\leq x_2\cup$&$3$\\
$x_1x_3=x_1x_2$&$x_2=x_3\leq x_1$&\\
\hline
$x_1x_2=x_2x_3$,&$\uparrow$&$3$\\
$x_2x_3=x_1x_2$&&\\
\hline
$x_1x_3=x_2x_3$,&$\uparrow$&$3$\\
$x_2x_3=x_1x_3$&&\\
\hline
$x_1x_2=x_1x_2x_3$,&$x_1\leq x_2\leq x_3\cup x_1\leq x_3\leq x_2\cup$&$4$\\
$x_1x_2x_3=x_1x_2$&$x_2\leq x_1\leq x_3\cup x_2\leq x_3\leq x_1$&\\
\hline
$x_1x_3=x_1x_2x_3$,&$\uparrow$&$4$\\
$x_1x_2x_3=x_1x_3$&&\\
\hline
$x_2x_3=x_1x_2x_3$,&$\uparrow$&$4$\\
$x_1x_2x_3=x_2x_3$&&\\
\hline
\end{tabular} 

\bigskip 

Notice that $\V(x_1= x_2\leq x_3)$ does not define an irreducible component for $Y=\V(x_1x_2=x_1x_3)$, since $\V(x_1= x_2\leq x_3)$ is included into the solution set of another irreducible component $\V(x_1\leq x_2\leq x_3)$. Similarly, $\V(x_3=x_1\leq x_2)$ is not an irreducible component for $Y$, since  
it is contained in the irreducible component $\V(x_1\leq x_3\leq x_2)$.

It turns out that the number of irreducible components does not depend on the semilattice  order $l$. One can directly compute the average number of irreducible components of algebraic sets defined by equations in three variables:
\begin{equation}
\overline{\Irr}(3)=\frac{6+2(2+2+2+2+2+2+3+3+3+4+4+4)}{25}=\frac{72}{25}=2.88
\label{eq:Irr(3,2)_handy}
\end{equation}

Recall that in Section~\ref{sec:average} we obtain the general expression for $\overline{\Irr}(n)$~(\ref{eq:Irr}). Clearly,~(\ref{eq:Irr}) will give~(\ref{eq:Irr(3,2)_handy}) for $n=3$.

\section{Decompositions of algebraic sets}
\label{sec:decomposition_properties}

Let $Y$ denote the solution set of an equation $t(X)=s(X)$ over the semilattice $L_l=\{a_1,a_2,\ldots,a_l\}$. The table above shows that any irreducible component sorts the variables $X$ into some order. The following definition formalizes this property of irreducible components.

Let $\s$ be a permutation of the set $\{1,2,\ldots,n\}$; $\s$ sorts the set $X$ as follows $\{x_{\s(1)},x_{\s(2)},\ldots,x_{\s(n)}\}$, i.e. $\s(i)$ is the $i$-th variable in the sorted set $X$. A permutation $\s$ is called a {\it a permutation of the first (second) kind} if $x_{\s(1)}\in\Var(t)\cap\Var(s)$ (respectively, $x_{\s(2)}\in\Var(t)\setminus\Var(s)$, $x_{\s(1)}\in\Var(s)\setminus\Var(t)$). Let $\chi(\s)\in\{1,2\}$ denote the kind of a permutation $\s$.

\begin{example}
Let us consider an algebraic set $Y_0=\V(x_1x_2=x_1x_3)$. By the table above, $Y_0$ is the union of the following irreducible components
\[
Y_1=\V(x_1\leq x_2\leq x_3),\;Y_2=\V(x_1\leq x_3\leq x_2),\;Y_3=\V(x_2=x_3\leq x_1)
\]
The irreducible components $Y_1,Y_2,Y_3$ define the following permutations 
\[
\s_1=\begin{pmatrix}
1&2&3\\
1&2&3
\end{pmatrix},
\s_2=\begin{pmatrix}
1&2&3\\
1&3&2
\end{pmatrix},
\s_3=\begin{pmatrix}
1&2&3\\
2&3&1
\end{pmatrix}.
\]
Moreover, $\s_1,\s_2$ are permutations of the first kind, whereas $\s_3$ is of the second kind. 
\label{ex:ex}
\end{example}

A permutation $\s$ defines an algebraic set $Y_\s$ as follows:
\begin{equation}
Y_\sigma=\V(\bigcup_{i=1}^{n-1}\{x_{\s(i)}
\leq x_{\s(i+1)}\})
\label{eq:Y_s_I}
\end{equation}
if $\chi(\s)=1$, and 
\begin{equation}
Y_\sigma=\V(\{x_{\s(1)}=x_{\s(2)}\}\bigcup_{i=2}^{n-1}\{x_{\s(i)}
\leq x_{\s(i+1)}\})
\label{eq:Y_s_II}
\end{equation}
if $\chi(\s)=2$.

\begin{example}
Let $\s_1,\s_2,\s_3$ be permutations from Example~\ref{ex:ex}. Obviously, the sets $Y_{\s_1},Y_{\s_2},Y_{\s_3}$ defined by~(\ref{eq:Y_s_I},\ref{eq:Y_s_II}) coincide with the sets $Y_1,Y_2,Y_3$ respectively.
\end{example}

\begin{lemma}
Let $\chi(\s)\in\{1,2\}$, then the set $Y_\sigma$ is irreducible and moreover 
\begin{equation}
\Gamma(Y_\s)\cong\begin{cases}
L_n,\mbox{ if $\chi(\s)=1$}\\
L_{n-1},\mbox{ if $\chi(\s)=2$}
\end{cases}
\label{eqq:Gamma}
\end{equation}
\label{l:Y_sigma_is_irreducible}
\end{lemma}
\begin{proof}
By the definition of a coordinate semilattice, $\Gamma(Y_\sigma)$ is generated by the elements $\{x_1,x_2,\ldots,x_n\}$ and has the following defined relations
\[
x_{\s(1)}\leq x_{\s(2)}\leq\ldots x_{\s(n)}\; \mbox{ if }\chi(Y_\s)=1
\]
and 
\[
x_{\s(1)}=x_{\s(2)}\leq\ldots x_{\s(n)}\; \mbox{ if }\chi(Y_\s)=2.
\]

Thus, $\Gamma(Y_\sigma)$ is a linearly ordered semilattice, and~(\ref{eqq:Gamma}) holds. By Proposition~\ref{pr:gamma_is_embedded_for_irr}, the set $Y_\sigma$ is irreducible.  
\end{proof}

The following lemma gives the irreducible decomposition of an algebraic set $Y=\V(t(X)=s(X))$. 

\begin{lemma}
An algebraic set $Y=\V(t(X)=s(X))$ is a union
\begin{equation}
\label{eq:union_of_Y}
Y=\bigcup_{\chi(\s)\in\{1,2\}} Y_\sigma.
\end{equation}
\label{l:union_of_Y}
\end{lemma}
\begin{proof}
Suppose $P=(p_1,p_2,\ldots,p_n)\in Y$. Let us sort $p_i$ in the ascending order
\[
p_{\s(1)}\leq p_{\s(1)}\leq\ldots\leq p_{\s(n)},
\]
where $\s$ is a permutation of the set $\{1,2,\ldots,n\}$. We have that $\s$ induces the sorting of the variable set $X$. Obviously, we may assume that $x_\s(1)\in\Var(t)$ (if $x_{\s(1)}\notin\Var(t)$, the properties of $L_l$ provides an existsence of a variable $x_\s(i)\in \Var(t)$ such that $p_{\s(i)}=p_{\s(1)}$; in this case one can swap the values $\s(1)$ and $\s(i)$).

\bigskip

For example, the point $P=(a_2,a_1,a_1)\in\V(x_1x_2=x_1x_3)$ defines $\s(1)=2$, $\s(2)=3$, $\s(3)=1$ (the permutation obtained equals $\s_3$ from Example~\ref{ex:ex}, so the point $(a_2,a_1,a_1)$ belongs to the set $Y_3$). 

\bigskip

Since $\s$ is defined by the inequalities between the coordinates $p_i$, it follows $P\in Y_\s$. 

Let us prove now $Y_\s\subseteq Y$ for each $\s$. Suppose $P=(p_1,p_2,\ldots,p_n)\in Y_\s$.  If $\chi(Y_\s)=1$ then 
\[
x_{\s(1)}\in\Var(t)\cap\Var(s)\Rightarrow t(P)=s(P)=p_{\s(1)}\Rightarrow P\in\V(t(X)=s(X)).
\]
Otherwise ($\chi(Y_\s)=2$),
$t(P)=p_{\s(1)}$, $s(P)=p_{\s(2)}$, and~(\ref{eq:Y_s_II}) gives $p_{\s(1)}=p_{\s(2)}$. Therefore $P\in\V(t(X)=s(X))$.
\end{proof}

\begin{lemma}
For distinct permutations $\sigma,\sigma^\pr$ we have $Y_{\sigma}\nsubseteq Y_{\sigma^\pr}$ in~(\ref{eq:union_of_Y}).
\label{l:about_point_P_sigma}
\end{lemma}
\begin{proof}
Let $\s$ be a permutation of the first or second kind, and $P_\s$ denote the following point
\[
p_{\s(i)}=a_i\mbox{ if }\chi(\s)=1,
\] 
and
\[
p_{\s(i)}=\begin{cases}
a_i,\; 2\leq i\leq n\\
a_2,\; i=1
\end{cases}
\;\mbox{ if $\chi(\s)=2$.}\]

For example, the permutations $\s_1,\s_2,\s_3$ from Example~\ref{ex:ex} define the points 
\[
P_1=(a_1,a_2,a_3),\; P_2=(a_1,a_3,a_2),\; P_3=(a_3,a_2,a_2),
\]
respectively.

Since $P_\s$ preserves the order of variables, we have $P_\sigma\in Y_\s$. 

Let us show now $P_\s\notin Y_{\s^\pr}$ for every $\s^\pr\neq \s$ (for example, each of the points $P_1,P_2,P_3$ above belong to a unique irreducible component from Example~\ref{ex:ex}:
\[
P_1\in Y_1\setminus(Y_2\cup Y_3),\;P_2\in Y_2\setminus(Y_1\cup Y_3),\;P_3\in Y_3\setminus(Y_1\cup Y_2)).
\]
There exists indexes $i<j$ such that $i=\s(\al)$, $j=\s(\beta)$, $i=\s^\pr(\al^\pr)$, $j=\s^\pr(\beta^\pr)$, with $\al<\beta$, $\al^\pr>\beta^\pr$. Hence the inequality $x_i\leq x_j$ holds in $Y_\s$, and $x_j\leq x_i$ holds in $Y_{\s^\pr}$. Let us consider the following two cases:

\begin{enumerate}
\item If $\chi(\s)=1$, then $p_i<p_j$ in $P_\s$, and we immediately obtain $P_\s\notin Y_{\s^\pr}$.  

\item Suppose $\chi(\s)=2$. One should assume that $p_i=p_j=a_2$ (if $p_i<p_j$ we immediately obtain $P_\s\notin Y_{\s^\pr}$). Then $\al=1$, $\beta=2$ and $i=\s(1)$, $j=\s(2)$ (one can similarly consider the case $i=\s(2)$, $j=\s(1)$). Hence $x_i\in\Var(t)\setminus\Var(s)$, $x_j\in\Var(s)\setminus\Var(t)$. By the definition of a permutation of the second kind, $\s^\pr(1)=k\neq j$, and the inequality $x_k\leq x_j$ holds in $Y_{\s^\pr}$.  Let $\gamma$ be the index such that $\s(\gamma)=k$. Since $\al=1$, $\beta=2$, we have $\gamma>2$. Then $p_k=a_{\gamma}$, and $p_j<p_k$ for $P_\s$. Thus, $P\notin Y_{\s^\pr}$.    
\end{enumerate}  
%If $p_i<p_j$ in $P_\s$ we immediately obtain $P_\s\notin Y_{\s^\pr}$. Thus, we assume that $p_i=p_j$ in $P_\s$ (this case is possible if $\chi(\s)=2$ and $i=\s(1)$, $j=\s(2)$, $p_i=p_j=a_2$). Therefore, the definition of a permutation of the second kind gives $i<j$.  

\end{proof}

According to Lemmas~\ref{l:Y_sigma_is_irreducible},~\ref{l:union_of_Y},~\ref{l:about_point_P_sigma}, we obtain the following statement.

\begin{theorem}
The union~(\ref{eq:union_of_Y}) is the irreducible decomposition of the set $Y=\V(t(X)=s(X))$. The number of irreducible components is equal to the number of permutations of the first and second kind.
\label{th:number_of_irr_compionents}
\end{theorem}

\section{Average number of irreducible components}
\label{sec:average}

One can directly compute that any $(k_1,k_2)$-equation admits 
\[
(n-k_1-k_2)(n-1)!
\]
permutations of the first kind and 
\[
k_1k_2(n-2)!
\]
permutations of the second kind. 

By Theorem~\ref{th:number_of_irr_compionents}, for a $(k_1,k_2)$-equation $t(X)=s(X)$ the number of its irreducible components equals  
\begin{equation}
\Irr(k_1,k_2,n)=(n-k_1-k_2)(n-1)!+k_1k_2(n-2)!
\label{eq:Irr(k_1,k_2,n,l)}
\end{equation}
The average number of irreducible components of algebraic sets defined by equations from $Eq(n)$ is
\begin{multline*}
\overline{\Irr}(n)=\frac{\sum_{(k_1,k_2)\in K_n}\#Eq(k_1,k_2,n)\Irr(k_1,k_2,n)}{\#Eq(n)}=\\
\frac{\sum_{k_1=0}^{n-1}\sum_{k_2=0}^{n-k_1}\#Eq(k_1,k_2,n)\Irr(k_1,k_2,n)-\#Eq(0,n,n)\Irr(0,n,n)}{\#Eq(n)}.
\end{multline*} 
Since 
\[
\Irr(0,n,n)=(n-0-n)(n-1)!+0n(n-2)!=0,
\]
we obtain 
\[
\overline{\Irr}(n)=\frac{\sum_{k_1=0}^{n-1}\sum_{k_2=0}^{n-k_1}\#Eq(k_1,k_2,n)\Irr(k_1,k_2,n)}{\#Eq(n)}.
\]
%Below we compute $\overline{\Irr}$ using the following binomial identities:

Below we compute $\overline{\Irr}$ using the following denotations:
\begin{enumerate}
\item $A\stackrel{(1)}{=}B$: an expression $B$ is obtained from $A$ by the binomial identity
\[
a\binom{n}{a}=n\binom{n-1}{a-1}
\]
\item $A\stackrel{(2)}{=}B$: an expression $B$ is obtained from $A$ by the following identity of binomial coefficients
\begin{equation}
\sum_{t=0}^n\binom{n}{t}t2^t=2n3^{n-1}.
\label{eq:identity}
\end{equation}
Let us demonstrate the proof of~(\ref{eq:identity}):
\begin{multline*}
\sum_{t=0}^n\binom{n}{t}t2^t\stackrel{(1)}{=}
n\sum_{t=0}^n\binom{n-1}{t-1}2^t=
2n\sum_{t=0}^n\binom{n-1}{t-1}2^{t-1}=
2n\sum_{u=0}^{n-1}\binom{n-1}{u}2^{u}=2n3^{n-1}
\end{multline*} 
\end{enumerate}

Let us compute $\overline{\Irr}(n)$. We have that 
\begin{multline*}
\sum_{k_1=0}^{n-1}\sum_{k_2=0}^{n-k_1}\#Eq(k_1,k_2,n)\Irr(k_1,k_2,n)=\\
\sum_{k_1=0}^{n-1}\sum_{k_2=0}^{n-k_1}\binom{n}{k_1}\binom{n-k_1}{k_2}
\left(n-k_1-k_2)(n-1)!+k_1k_2(n-2)!\right)=\\
n!\sum_{k_1=0}^{n-1}\sum_{k_2=0}^{n-k_1}\binom{n}{k_1}\binom{n-k_1}{k_2}-
(n-1)!\sum_{k_1=0}^{n-1}\sum_{k_2=0}^{n-k_1}\binom{n}{k_1}\binom{n-k_1}{k_2}k_1-\\
(n-1)!\sum_{k_1=0}^{n-1}\sum_{k_2=0}^{n-k_1}\binom{n}{k_1}\binom{n-k_1}{k_2}k_2+
(n-2)!\sum_{k_1=0}^{n-1}\sum_{k_2=0}^{n-k_1}\binom{n}{k_1}\binom{n-k_1}{k_2}k_1k_2
=\\S_1-S_2-S_3+S_4,
\end{multline*} 
where
\[
S_1=n!\sum_{k_1=0}^{n-1}\sum_{k_2=0}^{n-k_1}\binom{n}{k_1}\binom{n-k_1}{k_2}=n!\sum_{k_1=0}^{n-1}\binom{n}{k_1}2^{n-k_1}=n!(3^n-1),
\]
\begin{multline*}
S_2=(n-1)!\sum_{k_1=0}^{n-1}\sum_{k_2=0}^{n-k_1}\binom{n}{k_1}\binom{n-k_1}{k_2}k_1=
(n-1)!\sum_{k_1=0}^{n-1}\binom{n}{k_1}k_12^{n-k_1}\stackrel{(1)}{=}\\
n!\sum_{k_1=0}^{n-1}\binom{n-1}{k_1-1}2^{n-k_1}=
n!\sum_{t=0}^{n-2}\binom{n-1}{t}2^{n-1-t}=\\
n!\left(\sum_{t=0}^{n-1}\binom{n-1}{t}2^{n-1-t}-1\right)=
n!(3^{n-1}-1),
\end{multline*}

\begin{multline*}
S_3=(n-1)!\sum_{k_1=0}^{n-1}\sum_{k_2=0}^{n-k_1}\binom{n}{k_1}\binom{n-k_1}{k_2}k_2\stackrel{(1)}{=}\\
(n-1)!\sum_{k_1=0}^{n-1}\binom{n}{k_1}(n-k_1)\sum_{k_2=0}^{n-k_1}\binom{n-k_1-1}{k_2-1}=
(n-1)!\sum_{k_1=0}^{n-1}\binom{n}{k_1}(n-k_1)2^{n-k_1-1}=\\
(n-1)!\sum_{t=0}^n\binom{n}{t}t2^{t-1}=
\frac{(n-1)!}{2}\sum_{t=0}^n\binom{n}{t}t2^t\stackrel{(2)}{=}n!3^{n-1},
\end{multline*} 

\begin{multline*}
S_4=(n-2)!\sum_{k_1=0}^{n-1}\sum_{k_2=0}^{n-k_1}\binom{n}{k_1}\binom{n-k_1}{k_2}k_1k_2\stackrel{(1)}{=}\\
(n-2)!\sum_{k_1=0}^{n-1}\binom{n}{k_1}k_1(n-k_1)\sum_{k_2=0}^{n-k_1}\binom{n-k_1-1}{k_2-1}=
(n-2)!\sum_{k_1=0}^{n-1}\binom{n}{k_1}k_1(n-k_1)2^{n-k_1-1}=\\
\frac{(n-2)!}{2}\sum_{k_1=0}^{n}\binom{n}{k_1}k_1(n-k_1)2^{n-k_1}=
\frac{(n-2)!}{2}\sum_{t=0}^{n}\binom{n}{t}t(n-t)2^t=\\
\frac{(n-2)!}{2}\left(n\sum_{t=0}^{n}\binom{n}{k_1}t2^t-\sum_{t=0}^{n}\binom{n}{t}t^22^t\right)\stackrel{(2)}{=}
\frac{(n-2)!}{2}\left(2n^23^{n-1}-S_5\right),
\end{multline*} 
and
\begin{multline*}
S_5=\sum_{t=0}^{n}\binom{n}{k_1}t^22^t\stackrel{(1)}{=}
n\sum_{t=0}^{n}\binom{n-1}{t-1}t2^t=
n\left(\sum_{t=0}^{n}\binom{n-1}{t-1}(t-1)2^t+\sum_{t=0}^{n}\binom{n-1}{t-1}2^t\right)=\\
n\left(2\sum_{t=0}^{n}\binom{n-1}{t-1}(t-1)2^{t-1}+\sum_{t=0}^{n}\binom{n-1}{t-1}2^t\right)\stackrel{(2)}{=}
n\left(4(n-1)3^{n-2}+2\cdot 3^{n-1}\right)
\end{multline*} 

Finally, we obtain that
\begin{multline*}
S_1-S_2-S_3+S_4=n!(3^n-1)-n!(3^{n-1}-1)-n!3^{n-1}+\\
\frac{(n-2)!}{2}\left(2n^23^{n-1}-n(4(n-1)3^{n-2}+2\cdot 3^{n-1})\right)=
n!3^{n-1}+(n-2)!3^{n-2}n\left(3n-2(n-1)-3\right)=\\
n!3^{n-1}+n!3^{n-2}=4n!3^{n-2}
\end{multline*}
and
\begin{equation}
\label{eq:Irr}
\overline{\Irr}(n)=\frac{4n!3^{n-2}}{3^n-2}\sim\frac{4}{9}n!
\end{equation}

Notice that the final answer does not depend on $l$ if $l\leq n$. In particular,~(\ref{eq:Irr}) gives
\begin{equation}
\label{eq:Irr_3_2_from_formula}
\overline{\Irr}(3)=\frac{72}{25}=2.88
\end{equation}
for $n=3$, and~(\ref{eq:Irr_3_2_from_formula}) obviously coincides with~(\ref{eq:Irr(3,2)_handy}).

\end{document}